\documentclass[12pt]{amsart}
\usepackage{amsmath}
\usepackage{amsthm}
\usepackage{amssymb}
\usepackage{amscd}
\usepackage{amsfonts}
\usepackage{amsbsy}
\usepackage{graphicx}
\usepackage{tikz}

\newtheorem{theorem}{Theorem}
\newtheorem{remark}{Remark}
\newtheorem{proposition}{Proposition}
\newtheorem{lemma}{Lemma}
\newtheorem{corollary}{Corollary}
\newtheorem{definition}{Definition}

\def\ie{{\em i.e.,} }
\def\eg{{\em e.g.} }

\newfont\bbf{msbm10 at 12pt}

\def\phi{\varphi}
\def\R{{\mathbb R}}

\def\N{{\mathbb N}}
\def\Z{{\mathbb Z}}

\def\diam{\mbox{\rm diam}\,}

\def\theta{\vartheta}

\begin{document}

\title{Uncountably many planar embeddings of unimodal inverse limit spaces}

\author{Ana Anu\v{s}i\'c, Henk Bruin, Jernej \v{C}in\v{c}}
\address[A.\ Anu\v{s}i\'c]{Faculty of Electrical Engineering and Computing,
University of Zagreb,
Unska 3, 10000 Zagreb, Croatia}
\email{ana.anusic@fer.hr}
\address[H.\ Bruin]{Faculty of Mathematics,
University of Vienna,
Oskar-Morgenstern-Platz 1, A-1090 Vienna, Austria}
\email{henk.bruin@univie.ac.at}
\address[J.\ \v{C}in\v{c}]{Faculty of Mathematics,
University of Vienna,
Oskar-Morgenstern-Platz 1, A-1090 Vienna, Austria} 
\email{jernej.cinc@univie.ac.at}
\thanks{AA was supported in part by Croatian Science Foundation under the project IP-2014-09-2285.
HB and J\v{C} were supported by the FWF stand-alone project P25975-N25.
We gratefully acknowledge the support of the bilateral grant \emph{Strange Attractors and Inverse Limit Spaces},  \"Osterreichische
Austauschdienst (OeAD) - Ministry of Science, Education and Sport of the Republic of Croatia (MZOS), project number HR 03/2014.}
\date{\today}

\subjclass[2010]{37B45, 37E05, 54H20}
\keywords{unimodal map, inverse limit space, planar embeddings}

\begin{abstract}
For a point $x$ in the inverse limit space $X$ with a single unimodal bonding map 
we construct, with the use of symbolic dynamics, a planar embedding such 
that $x$ is accessible.
It follows that there are uncountably many non-equivalent planar
embeddings of $X$.
\end{abstract}

\maketitle

\section{Introduction}\label{sec:intro}

Inverse limit spaces can be often used as a model to construct attractors of some plane diffeomorphisms, see for example \cite{Wil1, Wil2, BaHo}. 
One of the simplest examples is the Knaster continuum (bucket handle), which is the attractor of Smale's horseshoe map and can be modeled as 
the inverse limit space of the full tent map $T_2(x):= \min \{2x, 2(1-x)\}$ for $x\in[0, 1]$, see \cite{Ba}. The topological unfolding of the Smale's horseshoe has been a topic of ongoing interest, related to the pruning front conjecture 
for the H\'{e}non family developed by Cvitanovi\'c et al. in \cite{Cv,CvGuPr} and demonstrated more recently with the work of Boyland, de Carvalho \& Hall  \cite{BoCaHa} and Mendoza \cite{Me}.

Inverse limit spaces with unimodal bonding maps
$T:[0,1] \to [0,1]$ (from now on denoted by $X$) are chainable.
The study of embeddings of chainable continua dates back to 1951 when Bing proved in \cite{Bing} that every chainable continuum can be embedded in the plane. However, his proof does not
offer any insight what such embeddings look like. The first explicit class of embeddings of $X$ was given by Brucks \& Diamond in \cite{BrDi}. Later, Bruin
\cite{Br1} extended this result showing that the embedding of $X$ can be made such that
the shift-homeomorphism extends to a 
Lipschitz map on $\R^2$. Both mentioned results are using symbolic dynamics 
as the main tool in the description of $X$. 

Locally, inverse limit spaces of unimodal maps roughly resemble Cantor sets of arcs. However, this is not true in general. 
In \cite{BBD} Barge, Brucks \& Diamond proved that in the tent family $T_s$
for a dense $G_{\delta}$ set of slopes $s\in [\sqrt{2},2]$, every open set of 
the inverse limit space $\underleftarrow{\lim}([s(1-\frac{s}{2}), \frac{s}{2}],T_s)$
not only contains a homeomorphic copy of the space itself but also 
homeomorphic copies of every  
inverse limit space of a tent map. Thus it would be interesting to see what kind of embeddings in the plane of complicated $X$ are possible.

Philip Boyland posed the following questions on the Continuum Theory and Dynamical Systems Workshop in Vienna in July 2015:
\begin{quote}
 Can a complicated $X$ be embedded in $\R^2$ in multiple ways? For example, do there exist embeddings in $\R^2$ of $X$ that 
 are non-equivalent to the standard embeddings constructed in \cite{BrDi} and 
 \cite{Br1}?
\end{quote}

For a special case of the full tent map
these two questions were already answered 
in the affirmative by Mayer \cite{May}, Mahavier \cite{Mah}, Schwartz \cite{Sch}
and  D\c{e}bski \& Tymchatyn \cite{DT}.

Let $K$ be a continuum.
A \emph{composant} $\mathcal{U}\subset K$ of a
point $x\in\mathcal{U}$ is the union of all proper subcontinua of $K$ containing point $x$.
We call $K$ {\em indecomposable} if it cannot be expressed as the union of two proper subcontinua.
In this case, $K$ has uncountably many composants and every composant is dense in $K$.

There are two fixed points of the map $T$: $0$ and $r$. 
Denote the composant of $(\ldots,0,0)$ by $\mathcal{C}$. 
We can decompose $X=\mathcal{C}\cup X'$, where $X'$ is the \emph{core} of $X$ and $\mathcal{C}$ compactifies on $X$, see e.g.\ \cite{InMa}. 
Let $\mathcal{R}$ be the arc-component of $(\ldots,r,r)$; if
$X'$ is indecomposable then $\mathcal R$ is indeed the composant of
$(\ldots,r,r)$ within $X'$.
The {\em standard embeddings} given in \cite{BrDi} and \cite{Br1} 
make $\mathcal{C}$ and $\mathcal{R}$ respectively accessible. 

\begin{definition}
A point $a\in X\subset\R^2$ is \emph{accessible (i.e., from the complement of $X$)} if there exists an arc $A=[x, y]\subset\R^2$ such that $a=x$ and $A\cap X=\{a\}$. 
We say that a composant $\mathcal{U}\subset X'$ or $\mathcal{U}=\mathcal{C}$ is \emph{accessible}, if $\mathcal{U}$ contains an accessible point.
\end{definition}

In the special case of the full tent map, i.e., with the Knaster continuum
$X$ as the inverse 
limit space, Mayer constructed uncountably many non-equivalent embeddings in \cite{May}.
Later, Mahavier showed in \cite{Mah} that for every composant $\mathcal{U}\subset X$, 
there exists a homeomorphism
$h:X\rightarrow \R^2$ such that each point of $h(\mathcal{U})$ is accessible. Schwartz extended Mahavier's result and
proved that embeddings of $X$ which do not make $\mathcal{C}$ or $\mathcal{R}$ accessible are non-equivalent to 
the standard embeddings.

\begin{definition}
	Denote two planar embeddings of $X$ by  $g_1\colon X\to E_1\subset \R^2$ and $g_2\colon X\to E_2\subset\R^2$. We say that $g_1$ and $g_2$ are \emph{equivalent embeddings} if there exists a homeomorphism $h\colon  E_1\to E_2$ which can be extended to the homeomorphism of the plane.
\end{definition}

In this paper we make use of symbolic dynamics description of $X$ introduced in \cite{BrDi} and \cite{Br1} 
and answer the questions of Boyland in the affirmative. 
We construct embeddings of $X$ by selecting (the itinerary of) the accessible point. For every unimodal map of positive topological entropy,
 we obtain uncountably many embeddings by making an arbitrary point 
from $X$ accessible. 

\begin{theorem}\label{thm:main}
	For every point $a\in X$ there exists an embedding of $X$ in the plane such that $a$ becomes accessible.
\end{theorem}

By $h_{top}(T)$ we denote the topological entropy of $T$. If $h_{top}(T)>0$, then the unimodal inverse limit space $X$ with the bonding map $T$ contains an indecomposable subcontinuum.

\begin{corollary}\label{cor:equiv}
Let $T$ be a unimodal interval map with $h_{top}(T)>0$.
Then there are uncountably many non-equivalent embeddings of $X$ in the plane.
\end{corollary}

The short outline of the paper is as follows. Section~\ref{sec:prelim} 
provides the basic set-up as introduced in \cite{BrDi} and \cite{Br1}. 
Next, we construct specific representations of $X$ in the 
plane in Section~\ref{sec:representation}. In Section~\ref{sec:embeddings} we prove that the representations given in Section~\ref{sec:representation} 
are indeed embeddings and prove the main results.
\\[5mm]
{\bf Acknowledgements:} We would like to thank the referee for the valuable comments.

\section{Preliminaries}\label{sec:prelim}
Let $\N:=\{1,2,3,\ldots\}$ and $\N_0:=\{0\}\cup\N$.
Let $T:[0,1] \to [0,1]$ be a unimodal map fixing $0$ and let $c$ denote the {\em critical point} of $T$. We define the \emph{inverse limit space with the bonding map $T$} by
$$
X:=\{(\ldots, x_{-2}, x_{-1}, x_0): T(x_{-(i+1)})=x_{-i}\in [0,1], \ i\in \N_0 \}\subset [0,1]^{-\N}
$$
equipped with the product metric 
$$
d((\ldots, x_{-2}, x_{-1}, x_0), (\ldots, y_{-2}, y_{-1}, y_0)):=\sum_{i\leq 0}2^i|x_i-y_i|.
$$
This makes $X$ a {\em continuum}, \ie a compact and connected metric space. Define the \emph{shift homeomorphism} as 
$\sigma\colon X\to X$, $\sigma((\ldots, x_{-2}, x_{-1}, x_0)):=(\ldots, x_{-2}, x_{-1}, x_0, T(x_0))$ and the \emph{projection map} as
$\pi_{n}\colon X\to [0,1]$, $\pi_n((\ldots, x_{-2}, x_{-1}, x_0)):=x_{-n}$, for every $n\in \N_0$.

In the construction of planar embeddings of spaces $X$ we recall a well-known symbolic description introduced in \cite{BrDi}. 
The space $X$ will be represented by the quotient space 
$\Sigma_{adm}/\!\!\sim$, where $\Sigma_{adm}\subseteq\{0, 1\}^{\Z}$ is equipped with the product topology.
We first need to recall the kneading theory for unimodal maps. 
To every $x\in [0,1]$ we assign its \emph{itinerary}: 
$$
I(x):=\nu_0(x)\nu_1(x)\ldots,$$ where $$\nu_i(x):=\left\{
\begin{array}{ll}
0, &  T^i(x)\in[0, c),\\
*, &  T^i(x)=c,\\
1, &  T^i(x)\in(c, 1] .
\end{array}
\right.
$$
Note that if $\nu_i(x)=*$ for some $i\in \N_0$, then $\nu_{i+1}(x)\nu_{i+2}(x)\ldots=I(T(c))$. The sequence $\nu:=I(T(c))$ is called the \emph{kneading sequence of $T$} 
and is denoted by $\nu=c_1 c_2\ldots$, where $c_{i}:=\nu_{i}(T(c))\in \{0, *, 1\}$ for every $i\in \N$. Observe that if $*$ appears in the kneading sequence, then $c$ is 
periodic under $T$, \ie there exists $n>0$ such that $T^n(c)=c$ and the kneading sequence is of the form $\nu=(c_1\ldots c_{n-1}*)^{\infty}$. 
In this case we adjust the kneading sequence by taking the smallest of $(c_1\ldots c_{n-1} 0)^{\infty}$ and $(c_1\ldots c_{n-1} 1)^{\infty}$ in the parity-lexicographical
ordering defined below. 

By $\#_1(a_1\ldots a_n)$ we denote the number of ones in a finite word $a_1\ldots a_n\in\{0, 1\}^n$; it can be either even or odd. 

Choose $t=t_1 t_2\ldots \in\{0, 1\}^{\N}$ and  $s=s_1 s_2\ldots \in\{0, 1\}^{\N}$ such that $s\neq t$.
Take the smallest $k\in \N$ such that $s_k\neq t_k$. Then the
\emph{parity-lexicographical ordering} is defined as
$$
s\prec t \Leftrightarrow\left\{
\begin{array}{ll}
s_k<t_k \text{ and } \#_1(s_1\ldots s_{k-1}) \text{ is even, or }\\
s_k>t_k \text{ and } \#_1(s_1\ldots s_{k-1}) \text{ is odd. }
\end{array}
\right.
$$
This ordering is also well-defined on $\{0, *, 1\}^{\N}$ once we define $0<*<1$.
\\
Thus if $(c_1\ldots c_{n-1} 0)^{\infty} \prec (c_1\ldots c_{n-1} 1)^{\infty}$ we modify $\nu=(c_1\ldots c_{n-1} 0)^{\infty}$, otherwise 
$\nu=(c_1\ldots c_{n-1} 1)^{\infty}$.
\\[2mm]
{\bf Example.} If $c$ is periodic of period $3$ then the kneading sequence for $T$ is $\nu=(10*)^{\infty}$. Since $101\ldots <100\ldots$ in parity-lexicographical ordering, we modify $\nu=(101)^{\infty}$.
\\[2mm]
In the same way we modify itinerary of an arbitrary point $x\in [0,1]$. If $\nu_i(x)=*$ and $i$ is the smallest positive integer with this property then we replace $\nu_{i+1}(x)\nu_{i+2}(x)\ldots$ with the modified kneading sequence. Thus $*$ can appear only once in the modified itinerary of an arbitrary point $x\in [0,1]$. 

From now onwards we assume that the itineraries of points from $[0,1]$ are modified.

It is a well-known fact (see \cite{MiTh}) that a kneading sequence completely characterizes the dynamics of unimodal map in the sense of the following proposition:

\begin{proposition}\label{admiss}
If a sequence $s_0s_1\ldots\in\{0, *, 1\}^{\N}$ is the itinerary of 
a point $x \in [T^2(c), T(c)]$, then
\begin{equation}\label{eq:adm}
I(T^2(c))\preceq s_ks_{k+1}\ldots\preceq \nu=I(T(c)), \textrm{ for every } k\in \N_0.
\end{equation}
Conversely, assume $s_0s_1\ldots\in\{0,*,1\}^{\N}$  satisfies \eqref{eq:adm}. 
If there exists $j\in \N_0$ such that 
$s_{j+1}s_{j+2}\ldots=\nu$, and $j$ is minimal
with this property, assume additionally that $s_j=*$.
Then $s_0s_1\ldots$ is realized as the itinerary of some $x\in [T^2(c), T(c)]$.
\end{proposition}

\begin{definition}
We say that a sequence $s_0s_1\ldots\in\{0, *, 1\}^{\N}$ is $\emph{admissible}$ if it is realized as the itinerary of some $x\in[0, 1]$. 
\end{definition}

\begin{remark}
 Note that the Proposition~\ref{admiss} gives conditions on admissible 
itineraries of points $x \in [T^2(c), T(c)]$. For points $y\in [0,T^2(c)) \cup (T(c),1]$ 
 admissible itineraries are exactly $0^{\N}$, $10^{\N}$, $0^{j}s_0s_{1}\ldots$, $10^{j-1}s_0s_{1}\ldots$ where $s_0s_1\ldots \in \{0,*,1\}^{\N}$ is the itinerary of the point $T^{j}(y)$ which satisfies the conditions of Proposition~\ref{admiss} for $j:=\min \{i\in \N: T^{i}(y)\in [T^{2}(c),T(c)]\}$.
\end{remark}

Next we show how to expand the above construction to $X$.  
Take $x=(\ldots, x_{-2}, x_{-1}, x_0)\in X$. Define the itinerary of $x$ as a two-sided infinite sequence 
$$\bar{I}(x):=\ldots\nu_{-2}(x)\nu_{-1}(x).\nu_0(x)\nu_1(x)\ldots\in\{0, *, 1\}^{\Z},$$ where $\nu_0(x)\nu_1(x)\ldots=I(x_0)$ and 
$$\nu_{i}(x)=\left\{
\begin{array}{ll}
0, &  x_i\in[0, c),\\
*, &  x_i=c,\\
1, &  x_i\in(c, 1],
\end{array}
\right.
$$
for all $i<0$. 

We make the same modifications as above. If $*$ appears for the first time at $\nu_k(x)$ for some $k\in\Z$, then $\nu_{k+1}(x)\nu_{k+2}(x)\ldots =\nu$. 
If there is no such minimal $k$, then the kneading sequence is periodic 
with a period $n\in\N$, $\nu=(c_1c_2\ldots c_{n-1}*)^{\infty}$ and the 
itinerary of $x$ is of the form $(c_1\ldots c_{n-1}*)^{\Z}$. 
Replace $(c_1\ldots c_{n-1}*)^{\Z}$ with the modified itinerary $(c_1c_2\ldots c_{n-1}c_n)^{\Z}$, where $\nu=(c_1\ldots$ $c_{n-1}c_n)^{\infty}$. In this way $*$ 
can appear at most once in every itinerary. Now we are ready to identify the inverse limit space with a quotient of 
a space of two-sided sequences consisting of two symbols.
\\
Let $\Sigma:=\{0, 1\}^{\Z}$ be the space of two-sided sequences equipped with the metric 
$$
d((s_i)_{i\in\Z},  (t_i)_{i\in\Z}):= \sum_{i\in\Z}\frac{|s_i-t_i|}{2^{|i|}},
$$
for $(s_i)_{i\in\Z}, (t_i)_{i\in\Z}\in \Sigma$. We define the shift homeomorphism 
$\sigma\colon\Sigma\to\Sigma$ as 
$$
\sigma(\ldots s_{-2}s_{-1}.s_0s_1\ldots):=\ldots s_{-2}.s_{-1}s_0s_1\ldots
$$
By $\Sigma_{adm}\subseteq\Sigma$ we denote all $s\in\Sigma$ such that either
\begin{itemize}
	\item[(a)] $s_ks_{k+1}\ldots$ is admissible for every $k\in\Z$, or
	\item[(b)] there exists $k\in\Z$ such that $s_{k+1}s_{k+2}\ldots=\nu$ and $s_{k-i}\ldots s_{k-1}*s_{k+1}s_{k+2}
	\ldots$ is admissible for every $i\in \N$.
\end{itemize} 
We abuse notation and call the two-sided sequences in $\Sigma_{adm}$ also \emph{admissible}.

Let us define an equivalence relation on the space $\Sigma_{adm}$. For sequences $s=(s_i)_{i\in\Z}$, $t=(t_i)_{i\in\Z} \in \Sigma_{adm}$ we define the relation
$$s\sim t \Leftrightarrow \left\{
\begin{array}{ll}
\textrm{ either } s_i=t_i \textrm{ for every } i\in \Z,\\
\textrm{ or if there exists } k\in\Z \textrm{ such that } s_i=t_i \textrm{ for all } i\neq k \textrm{ but } s_k\neq t_k \\
\textrm{ and } s_{k+1}s_{k+2}\ldots=t_{k+1}t_{k+2}\ldots=\nu.
\end{array}
\right.
$$
It is not difficult to see that this is indeed an equivalence relation on the space $\Sigma_{adm}$. 
Furthermore, every itinerary is identified with at most one different itinerary and the quotient space $\Sigma_{adm}/\!\!\sim$ of $\Sigma_{adm}$ is well defined.
It was also shown that $\Sigma_{adm}/\!\!\sim$ is homeomorphic to $X$. 
So in order to embed $X$ in the plane it is enough to embed $\Sigma_{adm}/\!\!\sim$ in the plane.  
For all observations in this paragraph we refer to the paper \cite{BrDi} of Brucks \& Diamond (Lemmas 2.2-2.4 and Theorem 2.5).
\\[2mm]
An \emph{arc} is a homeomorphic image of an interval $[a,b]\subset \R$.
A key fact for constructing embeddings is that $X$ is the union of 
{\em basic arcs} defined below.
Let $\overleftarrow{s}=\ldots s_{-2}s_{-1}.\in\{0, 1\}^{-\N}$
be an admissible left-infinite sequence (\ie every finite subword is admissible). The subset 
$$
A(\overleftarrow{s}):=\{x\in X : \nu_i(x)=s_i, \forall i<0\} \subset X
$$
is called a {\em basic arc}. Note that
$\pi_0: A(\overleftarrow{s}) \to [0,1]$ is injective. 
In \cite[Lemma 1]{Br1} it was observed that $A(\overleftarrow{s})$ is indeed an arc (possibly degenerate).
For every basic arc we define two quantities as follows:
\begin{eqnarray*}
\tau_L(\overleftarrow{s})
&:=&\sup\{n>1 : s_{-(n-1)}\ldots s_{-1}=c_1c_2\ldots c_{n-1}, \#_1(c_1\ldots c_{n-1})\textrm{ odd}\}, \\
\tau_R(\overleftarrow{s})
&:=&\sup\{n\geq 1 : s_{-(n-1)}\ldots s_{-1}=c_1c_2\ldots c_{n-1}, \#_1(c_1\ldots c_{n-1})\textrm{ even}\}.
\end{eqnarray*}

\begin{remark}
For $n=1$, $c_1c_2\ldots c_{n-1}=\emptyset$ and $\#_1(\emptyset)$ is even. Thus $\tau_R(\overleftarrow{s})=1$ if and only if $s_{-(n-1)}\ldots s_{-1}\neq c_1c_2\ldots c_{n-1}$ for all $n>1$.
\end{remark}

These definitions first appeared in \cite{Br1} in order to study the number of endpoints of inverse limit spaces $X$.
We now adapt two lemmas from \cite{Br1} that we will use later in the paper.

\begin{lemma}\label{lem:first}(\cite{Br1}, Lemma 2)
Let $\overleftarrow{s} \in \{ 0,1\}^{-\N}$ be admissible 
such that $\tau_L(\overleftarrow{s}), \tau_R(\overleftarrow{s})<\infty$.
Then 
$$
\pi_0(A(\overleftarrow{s}))=[T^{\tau_L(\overleftarrow{s})}(c), T^{\tau_R(\overleftarrow{s})}(c)].
$$
If $\overleftarrow{t}\in\{0, 1\}^{-\N}$ is another
admissible left-infinite sequence such that
 $s_i=t_i$ for all $i<0$ except 
for $i=-\tau_R(\overleftarrow{s})=-\tau_R(\overleftarrow{t})$ (or $i=-\tau_L(\overleftarrow{s})=-\tau_L(\overleftarrow{t})$), 
then $A(\overleftarrow{s})$ and $A(\overleftarrow{t})$ have a common boundary point.
\end{lemma}

\begin{lemma}\label{lem:second}(\cite{Br1}, Lemma 3)
If $\overleftarrow{s} \in \{ 0,1\}^{-\N}$ is admissible, then 
\begin{eqnarray*}
\sup\pi_0(A(\overleftarrow{s}))=\inf\{T^n(c):  s_{-(n-1)}\ldots s_{-1}=c_1\ldots c_{n-1}, n \geq 1, \#_1(c_1\ldots c_{n-1})\textnormal{ even}\},\\
\inf\pi_0(A(\overleftarrow{s}))=\sup\{T^n(c):  s_{-(n-1)}\ldots s_{-1}=c_1\ldots c_{n-1},n \geq 1, \#_1(c_1\ldots c_{n-1})\textnormal{ odd}\}.
\end{eqnarray*}
\end{lemma}

{\bf Example.} Take the unimodal map with the kneading sequence $\nu=(101)^{\infty}$. 
Then $\overleftarrow{s}= (011)^{\infty}010.$ and $\overleftarrow{t}=(011)^{\infty}110.$ are admissible, 
$\tau_L(\overleftarrow{s})=\tau_L(\overleftarrow{t})=3$, $\tau_R(\overleftarrow{s})=\tau_R(\overleftarrow{t})=1$ and $s_i=t_i$ for all $i<0$ except for
$i=-3= -\tau_L(\overleftarrow{s})=-\tau_L(\overleftarrow{t})$. By Lemma~\ref{lem:second}, $\pi_0(A(\overleftarrow{s})) = \pi_0(A(\overleftarrow{t})) = 
[T^3(c),T(c)]$, and by Lemma~\ref{lem:first}, $A(\overleftarrow{s})$ and $A(\overleftarrow{t})$ have a common boundary
point which is projected to $T^3(c)$, see Figure~\ref{pic1}.
Note that in this example both $\tau_L$ and $\tau_R$ agree for $\overleftarrow{s}$ and $\overleftarrow{t}$,
which need not be the case in general.

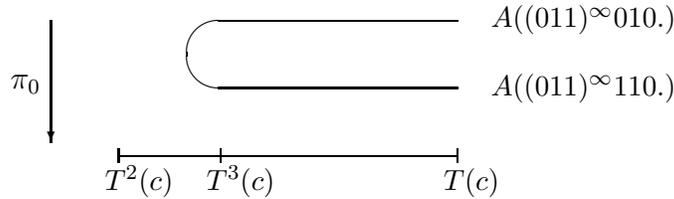
\begin{figure}[!ht]
\unitlength=9mm
\begin{picture}(10,4)(0,-1)

\put(1,2){\vector(0,-1){1.8}}
\put(0.4,1){$\pi_0$}

\put(2,0){\line(1,0){5}}
\put(3.5,1){\line(1,0){3.5}}
\put(3.5,2){\line(1,0){3.5}}
\put(3.5,1.5){\oval(1,1)[l]}

\put(2,-0.1){\line(0,1){0.2}}
\put(1.8,-0.5){\small $T^2(c)$}
\put(3.5,-0.1){\line(0,1){0.2}}
\put(3.3,-0.5){\small $T^3(c)$}
\put(7,-0.1){\line(0,1){0.2}}
\put(6.8,-0.5){\small $T(c)$}

\put(7.5,0.9){\small $A((011)^{\infty}110.)$}
\put(7.5,1.9){\small $A((011)^{\infty}010.)$}
\end{picture}
\caption{Example of two basic arcs having a boundary point in common.}
\label{pic1}
\end{figure}

\section{Representation in the plane}\label{sec:representation}

This section is the first step towards embedding $X$ in the plane so that an arbitrary point $a\in X$ becomes accessible. We denote the symbolic representation of $a$ by $\ldots l_{-2}l_{-1}.l_0l_1\ldots:=\bar{I}(a)$, so $a\in A(\ldots l_{-2}l_{-1}.)$.
We present the following ordering on $\{0, 1\}^{-\N}$ 
depending on some $L=\ldots l_{-2}l_{-1}.$ and we work with this ordering from now onwards. 

\begin{definition}
Let $\overleftarrow{s}, \overleftarrow{t}\in\{0, 1\}^{-\N}$ and let $k\in \N$ be the smallest natural number such that $s_{-k}\neq t_{-k}$. Then 
$$
\overleftarrow{s}\prec_L\overleftarrow{t} \Leftrightarrow 
\begin{cases}
t_{-k}=l_{-k} \text{ and } \#_1(s_{-(k-1)}\ldots s_{-1})-\#_1(l_{-(k-1)}\ldots l_{-1}) \text{ even, or }\\
s_{-k}=l_{-k} \text{ and } \#_1(s_{-(k-1)}\ldots s_{-1})-\#_1(l_{-(k-1)}\ldots l_{-1}) \text{ odd.}
\end{cases}
$$
\end{definition}

Note that such ordering is well-defined and the left infinite tail $L$ is the largest sequence.

\begin{lemma}\label{lem:third}
Assume $\overleftarrow{s}\prec_L\overleftarrow{u}\prec_L\overleftarrow{t}$ and assume that $s_{-n}\ldots s_{-1}=t_{-n}\ldots t_{-1}$. 
Then also $u_{-n}\ldots u_{-1}=s_{-n}\ldots s_{-1}=t_{-n}\ldots t_{-1}$.
\end{lemma}

\begin{proof}
If $n=1$ the statement follows easily so let us assume that $n\geq 2$.
Assume that there exists $k<n$ such that $u_{-k}\neq s_{-k}$ and take $k$ the smallest natural number with this property. 
Assume without loss of generality that $(-1)^{\#_1(s_{-(k-1)}\ldots s_{-1})}=(-1)^{\#_1(l_{-(k-1)}\ldots l_{-1})}$ (the proof follows similarly when the parities are different).
Since $\overleftarrow{s}\prec_L\overleftarrow{u}$ it follows that $u_{-k}=l_{-k}$. Also,  $\overleftarrow{u}\prec_L\overleftarrow{t}$ gives $t_{-k}=l_{-k}$.
Since $u_{-k}\neq t_{-k}$, we get a contradiction. 
\end{proof}

Let $C \subset [0,1]$ be the middle-third Cantor set, 
$$C:=[0,1]\setminus \bigcup^{\infty}_{m=1} \bigcup^{3^{m-1}-1}_{k=0}(\frac{3k+1}{3^{m}},\frac{3k+2}{3^m}).$$

Points in $C$ are coded by the left-infinite sequences of zeros and ones. 
We embed basic arcs in the plane as horizontal lines along the Cantor set 
and then join corresponding endpoints with semi-circles as in Figure~\ref{pic1}. The ordering has to be defined in a way that semi-circles neither cross horizontal lines nor each other. 

{\bf Example.} For $L=1^{\infty}.$, points in $C$ are coded as in Figure~\ref{pic2} $(a)$. Note that this is the same ordering as in the paper by Bruin \cite{Br1}. 
The ordering obtained by $L=0^{\infty}1.$ is the ordering from the paper by Brucks \& Diamond \cite{BrDi}. In Figure~\ref{pic2} $(b)$ points in Cantor set are coded with  
respect to $L=\ldots 101.$.

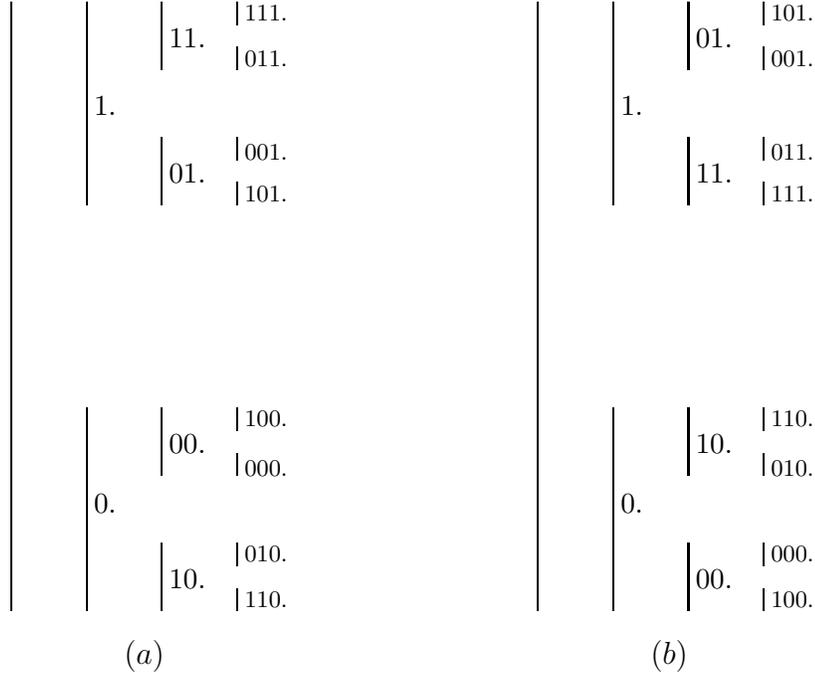
\begin{figure}[!ht]
	\unitlength=1mm
	
	\begin{picture}(85,87)(35,-6)
	\put(20,0){\line(0,1){81}}
	\put(30,0){\line(0,1){27}}
	\put(30,54){\line(0,1){27}}
	\put(40,0){\line(0,1){9}}
	\put(40,18){\line(0,1){9}}
	\put(40,54){\line(0,1){9}}
	\put(40,72){\line(0,1){9}}
	\put(50,0){\line(0,1){3}}
	\put(50,6){\line(0,1){3}}
	\put(50,18){\line(0,1){3}}
	\put(50,24){\line(0,1){3}}
	\put(50,54){\line(0,1){3}}
	\put(50,60){\line(0,1){3}}
	\put(50,72){\line(0,1){3}}
	\put(50,78){\line(0,1){3}}	
	\put(31,13){\small $0.$}
	\put(31,66){\small $1.$}
	\put(41,3){\small $10.$}
	\put(41,21){\small $00.$}
	\put(41,57){\small $01.$}
	\put(41,75){\small $11.$}
	\put(51,0.5){\scriptsize $110.$}
	\put(51,6.5){\scriptsize $010.$}
	\put(51,18){\scriptsize $000.$}
	\put(51,24.5){\scriptsize $100.$}
	\put(51,54.5){\scriptsize $101.$}
	\put(51,60){\scriptsize $001.$}
	\put(51,72.5){\scriptsize $011.$}
	\put(51,78.5){\scriptsize $111.$}
	\put(35, -7){$(a)$}
	
	\put(90,0){\line(0,1){81}}
	\put(100,0){\line(0,1){27}}
	\put(100,54){\line(0,1){27}}
	\put(110,0){\line(0,1){9}}
	\put(110,18){\line(0,1){9}}
	\put(110,54){\line(0,1){9}}
	\put(110,72){\line(0,1){9}}
	\put(120,0){\line(0,1){3}}
	\put(120,6){\line(0,1){3}}
	\put(120,18){\line(0,1){3}}
	\put(120,24){\line(0,1){3}}
	\put(120,54){\line(0,1){3}}
	\put(120,60){\line(0,1){3}}
	\put(120,72){\line(0,1){3}}
	\put(120,78){\line(0,1){3}}
	\put(101,13){\small $0.$}
	\put(101,66){\small $1.$}
	\put(111,3){\small $00.$}
	\put(111,21){\small $10.$}
	\put(111,57){\small $11.$}
	\put(111,75){\small $01.$}
	\put(121,0.5){\scriptsize $100.$}
	\put(121,6.5){\scriptsize $000.$}
	\put(121,18){\scriptsize $010.$}
	\put(121,24.5){\scriptsize $110.$}
	\put(121,54.5){\scriptsize $111.$}
	\put(121,60){\scriptsize $011.$}
	\put(121,72.5){\scriptsize $001.$}
	\put(121,78.5){\scriptsize $101.$}
	\put(105, -7){$(b)$}
	\end{picture}
	\caption{Coding the Cantor set with respect to $(a)$ $\protect L=\ldots 111.$ and $(b)$ $\protect L=\ldots 101.$}
	\label{pic2}
\end{figure}

From now onwards, we assume that $\overleftarrow{s} \in \{0,1\}^{\N}$ is an admissible  left-infinite sequence.
Define $\psi_L: \{ 0 , 1 \}^{-\N}\to C$ as
$$
\psi_L(\overleftarrow{s}) := \sum_{i=1}^\infty (-1)^{\#_1(l_{-i}\ldots l_{-1})-\#_1(s_{-i}\ldots s_{-1})} 3^{-i} + \frac{1}{2} ,
$$
and we let $C_{adm} := \{\psi_L(\overleftarrow{s}) : \overleftarrow{s} \text{ admissible left-infinite sequence}\}$
be the subset of ``admissible vertical coordinates''.
Note that $\psi_L(L) = 1$ is the largest point in $C_{adm}$.

From now onwards let $d_e$ denote the Euclidean distance in $\R^{2}$.

\begin{remark}\label{rem:distance}
Note that if $\overleftarrow{s}, \overleftarrow{t}\in \{ 0,1 \}^{-\N}$ are such that 
$s_{-n}\ldots s_{-1}=t_{-n}\ldots t_{-1}$, then $d_e(\psi_L(\overleftarrow{s}), 
\psi_L(\overleftarrow{t}))\leq 3^{-n}$. If $s_{-n}\neq t_{-n}$, 
then $d_e(\psi_L(\overleftarrow{s}), \psi_L(\overleftarrow{t})) \geq 3^{-n}$.
\end{remark}

Now we represent $X$ as the quotient space of the subset of $I \times C_{adm}$
for $I := [0,1]$. To every point $x=(\ldots, x_{-2}, x_{-1}, x_0)\in X$ we will
assign either a point or two points in $I \times C_{adm}$ by rule 
\eqref{eq:relation} below. 
From now on, write $\dots s_{-3} s_{-2} s_{-1}. := \dots \nu_{-3}(x) \nu_{-2}(x) \nu_{-1}(x)$.
Let $\phi: X \rightarrow I\times C_{adm}$ be defined in the following way:
\begin{equation}\label{eq:relation}
\phi(x):=\begin{cases}
(x_0, \psi_{L}((s_i)_{i<0})), & \text{ if } s_i \neq * \text{ for every } i<0,\\
(x_0, p)\cup (x_0, q), & \text{ if } s_i = * \text{ for some } i<0,
\end{cases}
\end{equation}
where
$$
\begin{cases}
p=\psi_L(\ldots s_{-(i+1)} 0 s_{-(i-1)} \ldots s_{-1}.),\\
q=\psi_L(\ldots s_{-(i+1)} 1 s_{-(i-1)} \ldots s_{-1}.).
\end{cases}
$$
Set $Y:=\phi(X) \subset I\times C_{adm}$. The next step is to identify points in $Y$ in the same way as they are identified in the symbolic representation of $X$. 
For $a,b\in Y$:
$$
a\sim b \text{ if there exists } x\in X \text{ such that } a, b\in\phi(x).
$$
If $a\neq b\sim a$ we write $\tilde{a}:=b$.
If $\tilde{a}=b$ and $x\in X$ is such that $a, b\in\phi(x)$ and $s_{-i}=*$ we say that $a$ and $b$ are \emph{joined at level $i$}.

Note that $\phi\colon X\to Y/\!\!\sim$ is a well-defined map. Equip $Y$ with the Euclidean topology and $Y/\!\!\sim$ with the standard quotient topology. 
 Let $\pi_C:I\times C\to C$ and $\pi_I\colon I\times C\to I$ denote the
\emph{natural projections}.
The next proposition is an analogue of Proposition $4$ from \cite{Br1}. We prove it here for the sake of completeness. 

\begin{proposition}\label{prop:homeo}
The map $\phi\colon X\to Y/\!\!\sim$ is a homeomorphism.
\end{proposition}

\begin{proof}
We first prove that $Y/\!\!\sim$ is a Hausdorff space and because $X$ is compact it is enough to check that $\phi$ is a continuous bijection to obtain a homeomorphism
between $X$ and $Y/\!\!\sim$, see \eg Theorem 26.6. in \cite{Mu}.
	
Take $x\neq y\in Y$ such that $x\neq\tilde{y}$. First assume that $|\pi_{I}(x)-\pi_{I}(y)|=0$.	
	Let $\delta:=\min\{ |\pi_{C}(x)-\pi_{C}(y)|,|\pi_{C}(\tilde{x})-\pi_{C}(y)| \}$. Then take $\{z : | \pi_{C}(x)-\pi_{C}(z)| \text{ or }
	|\pi_{C}(\tilde{x})-\pi_{C}(z)| < \delta/3 \}$  and $\{z : | \pi_{C}(y)-\pi_{C}(z)| \text{ or } |\pi_{C}(\tilde{y})-\pi_{C}(z)|< \delta/3  \}$ 
	for open neighbourhoods of $x$ and $y$ respectively and they are disjoint. Now assume that $\varepsilon:=|\pi_{I}(x)-\pi_{I}(y)|>0$. 
	Then $\{ z : |\pi_{I}(x)-\pi_{I}(z)|<\varepsilon/3 \}$ and $\{ z : |\pi_{I}(y)-\pi_{I}(z)|<\varepsilon/3 \}$ are disjoint open neighbourhoods for $x$ and $y$ respectively,
	so $Y/\!\!\sim$ is indeed a Hausdorff space.
	
Now we prove that $\phi$ is continuous. It is enough to prove that for $a\in X$ and a sequence $(x^n)_{n\in \N}\subset X$  such that $\lim_{n\rightarrow \infty} x^n=a$ it holds that $\lim_{n\rightarrow \infty} \phi(x^n)=\phi(a)$. 
Assume that $\lim_{n\rightarrow \infty} x^n=a$. Thus for every $M\in \N$ there exists $N\in \N$ such that for every $n\geq N$ it follows that 
$\nu_{-M}(x^n)\ldots \nu_{M}(x^n)=\nu_{-M}(a)\ldots \nu_{M}(a)$. We need to show that for every open $\phi(a)\in U\subset Y/\!\!\sim$ there exists $N'\in \N$ such that for 
	every $n\geq N'$ it holds that $\phi(x^n)\in U$. Let us fix an open set $U\ni \phi(a)$. If for $x=(\ldots, x_{-1}, x_0)\in X$ 
	there exists $i\in \N$ such that $\nu_{-i}(x)=*$ then we set $\phi(x)=\phi'(x)\cup \phi''(x)$ where 
	$\phi'(x):=(x_0,\psi_L(\ldots \nu_{-(i+1)}(x)0 \nu_{-(i-1)}(x)\ldots \nu_{-2}(x)\nu_{-1}(x).))$ and  
	$\phi''(x):=(x_0,\psi_{L}(\ldots \nu_{-(i+1)}(x)1 \nu_{-(i-1)}(x)\ldots \nu_{-2}(x)\nu_{-1}(x).))$.
	
\textbf{Case I:} For every $i\in \N$, $\nu_{-i}(a)\neq *$. If there exists $K\in \N$ such that for every $n\geq K$ it follows that $\nu_{-j}(x^n)\neq *$ for every $j\in \N$,
	then there is $N'\geq K$ such that $\phi(x^{n})\in U$ for every $n\geq N'$. \\
        Now assume that there exists an increasing sequence $(n_i)_{i\in \N}\subset \N$ such that $\nu_{-j}(x^{n_i})=*$ for some $j\in\N$. Then there exist open sets
        $U_1^{n_i},U_2^{n_i}\subset Y$ such that $\phi'(x^{n_i})\in U_1^{n_i}$ and $\phi''(x^{n_i})\in U_2^{n_i}$ and $\phi^{-1}(U)=U_1^{n_i}\cup U_2^{n_i}$ 
        for every $i\in \N$. Because $x^n\rightarrow a$ as $n\rightarrow \infty$, by the definition of $\phi$ it follows that $\phi'(x^{n_i})\to\phi(a)$ and $\phi''(x^{n_i})\to\phi(a)$ 
         as $i\rightarrow \infty$. Thus we again conclude that there exists $N'\in \N$ such that for every $n\geq N'$
        it follows that $\phi(x^{n})\in U$.

\textbf{Case II:} Let $K\in \N$ be such that $\nu_{K}(a)=*$ and thus $\phi(a)=\phi'(a)\cup \phi''(a)$. 
	Take $M>K$ so that $\nu_{-M}(a)\ldots \nu_{M}(a)=\nu_{-M}(x^n)\ldots\nu_{-K}(x^n)\ldots\nu_{K}(x^n)\ldots \nu_{M}(x^n)$  for every $n\geq N$, and so $\phi(x^n)=\phi'(x^n)\cup \phi''(x^n)$. 
	Thus there exist open sets $U_1,U_2\subset Y$ such that $\phi'(a)\in U_1$, $\phi''(a)\in U_2$ and $\phi^{-1}(U)=U_1\cup U_2$. It follows that there exists $N'>N$ 
	such that for every $n>N'$ it holds that $\phi'(x^n)\in U_1$ and $\phi''(x^n)\in U_2$ and thus $\phi(x^n)\in U$.
\end{proof}	

Now we are ready to represent $X$ in the plane. This is still not an embedding but it is the first step towards it. Connect identified points in $I\times C_{adm}$ with semi-circles. 
Suppose $a\neq b\in Y$ are joined at level $n$. By Lemma~\ref{lem:first}, points $a$ and $b$ are both endpoints of
basic arcs in $I\times C_{adm}$ and are both 
right or left endpoints. If $\#_1(c_1\ldots c_{n-1})$ is even (odd), $a$ and $b$ are right (left) endpoints and we join 
them with a semi-circle on the right (left), see Figure~\ref{pic1}.

\begin{proposition}\label{prop:cross}
Every semi-circle defined above crosses neither $Y$ nor another semi-circle.
\end{proposition}
\begin{proof}
\textbf{Case I:} Assume that there is a semi-circle oriented to the right which intersects an arc $A$ in $Y$. (See Figure~\ref{case1}.)

\begin{figure}[!ht]
\centering
\begin{tikzpicture}
\draw (1,1)--(9,1);
\draw (1,3)--(9,3);
\draw[domain=270:450] plot ({9+cos(\x)},{2+sin(\x)});
\draw (3,2)--(12,2);
\node at (2.5,2) {\small $A$};
\end{tikzpicture}
\caption{Case I in the proof of Proposition~\ref{prop:cross}.}
\label{case1}
\end{figure}
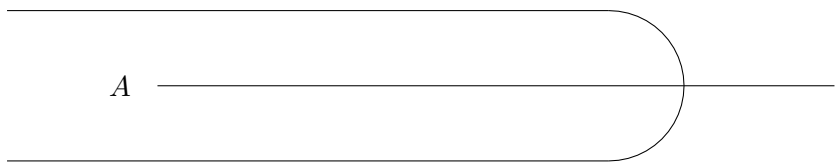

 Translated to symbolics, this means that there exist $n\in\N$ and $\overleftarrow{s}\prec_L\overleftarrow{u}\prec_L\overleftarrow{t}$ such that $s_{-(n-1)}\ldots 
 s_{-1}=t_{-(n-1)}\ldots t_{-1}=c_1\ldots c_{n-1}$, $s_{-n}\neq t_{-n}$ and $\#_1(c_1\ldots c_{n-1})$ is even. By Lemma~\ref{lem:third},  
 $u_{-(n-1)}\ldots u_{-1}=c_1\ldots c_{n-1}$. By Lemma~\ref{lem:second} it follows that $\text{sup} \{\pi_{I}(A)\}\leq T^n(c)$, and thus an intersection between the arc 
 $A$ and a semi-circle cannot occur.\\
\\
\textbf{Case II:} Assume that we have a crossing of two semi-circles which project to the same point in $I$. (See Figure~\ref{case2}.) 

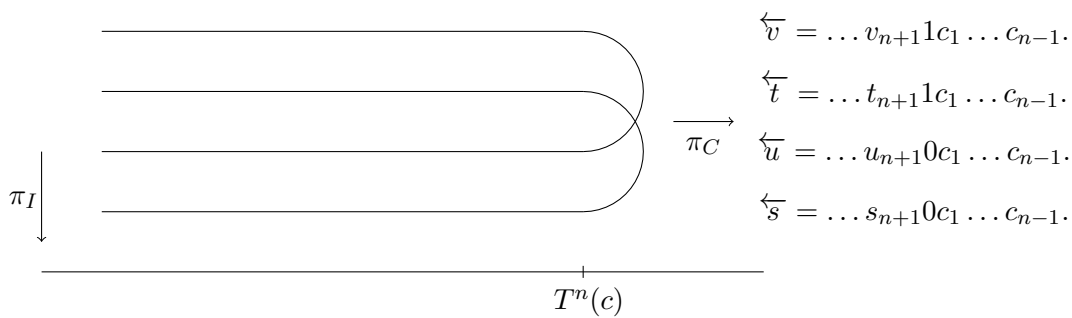
\begin{figure}[!ht]
\centering
\begin{tikzpicture}[scale=0.8]
\draw (1,1)--(9,1);
\draw (1,3)--(9,3);
\draw[domain=270:450] plot ({9+1*cos(\x)},{2+1*sin(\x)});
\draw (1,2)--(9,2);
\draw (1,4)--(9,4);
\draw[domain=270:450] plot ({9+1*cos(\x)},{3+1*sin(\x)});
\draw(0,0)--(12,0);
\draw (9,-0.1)--(9,0.1);
\node at (9.1,-0.5) {\small $T^n(c)$};
\draw[->] (0,2)--(0,0.5);
\node at (-0.3,1.2) {\small $\pi_I$};
\draw[->] (10.5,2.5)--(11.5,2.5);
\node at (11, 2.1) {\small $\pi_C$};
\node at (14.5,4) {\small $\overleftarrow{v}=\ldots v_{n+1} 1c_1\ldots c_{n-1}.$};
\node at (14.5,3) {\small $\overleftarrow{t}=\ldots t_{n+1} 1c_1\ldots c_{n-1}.$};
\node at (14.5,2) {\small $\overleftarrow{u}=\ldots u_{n+1} 0c_1\ldots c_{n-1}.$};
\node at (14.5,1) {\small $\overleftarrow{s}=\ldots s_{n+1} 0c_1\ldots c_{n-1}.$};
\end{tikzpicture}
\caption{Case II in the proof of Proposition~\ref{prop:cross}.}
\label{case2}
\end{figure}

Assume that there exist $n\in\N$ and $\overleftarrow{s}\prec_L\overleftarrow{u}\prec_L\overleftarrow{t}\prec_L\overleftarrow{v}$ such that $s_i=t_i$ for all $i<0$ except for $i=-n$
and $s_{-(n-1)}\ldots s_{-1}=t_{-(n-1)}\ldots t_{-1}=c_1\ldots c_{n-1}$ and $u_i=v_i$ for all $i<0$ except for $i=-n$ and
$u_{-(n-1)}\ldots u_{-1}=v_{-(n-1)}\ldots v_{-1}=c_1\ldots c_{n-1}$. If $s_{-n}=v_{-n}$, then by Lemma~\ref{lem:third} also $t_{-n}=u_{-n}=s_{-n}=v_{-n}$
which contradicts the assumption. It follows that $s_{-n}\neq v_{-n}$, because $\overleftarrow{v},\overleftarrow{u}$ and
$\overleftarrow{t},\overleftarrow{s}$ are respectively connected by a right semi-circle. Assume without loss of generality that $v_{-n}=1$ and $s_{-n}=0$. This gives $t_{-n}=1$ and $u_{-n}=0$. \\
Now take the smallest integer $m>n$ such that $v_{-m}\neq t_{-m}$; this $m$ is also the smallest integer such that $u_{-m}\neq s_{-m}$.
By the previous paragraph, if $(-1)^{\#_1(s_{-(m-1)}\ldots s_{-1})}=(-1)^{\#_1(u_{-(m-1)}\ldots u_{-1})}$ $\neq (-1)^{\#_1(t_{-(m-1)}\ldots t_{-1})}=(-1)^{\#_1(v_{-(m-1)}\ldots v_{-1})}$,
the possibilities for 
$s_{-m},$ $u_{-m}, t_{-m}, v_{-m}$ are (depending on the parities of ones): 
(1) $s_{-m}=0, u_{-m}=1, t_{-m}=1, v_{-m}=0$, or
(2) $s_{-m}=1, u_{-m}=0, t_{-m}=0, v_{-m}=1$.
Both cases lead to a contradiction with $s_{-m}=t_{-m}$ and $u_{-m}=v_{-m}$.
\end{proof}

Thus our ordering gives a representation $Y\cup \{\text{semi-circles}\}$ of $X$ in the plane.
Figure~\ref{representation1} and Figure~\ref{representation2} give
two examples of these planar representations. 

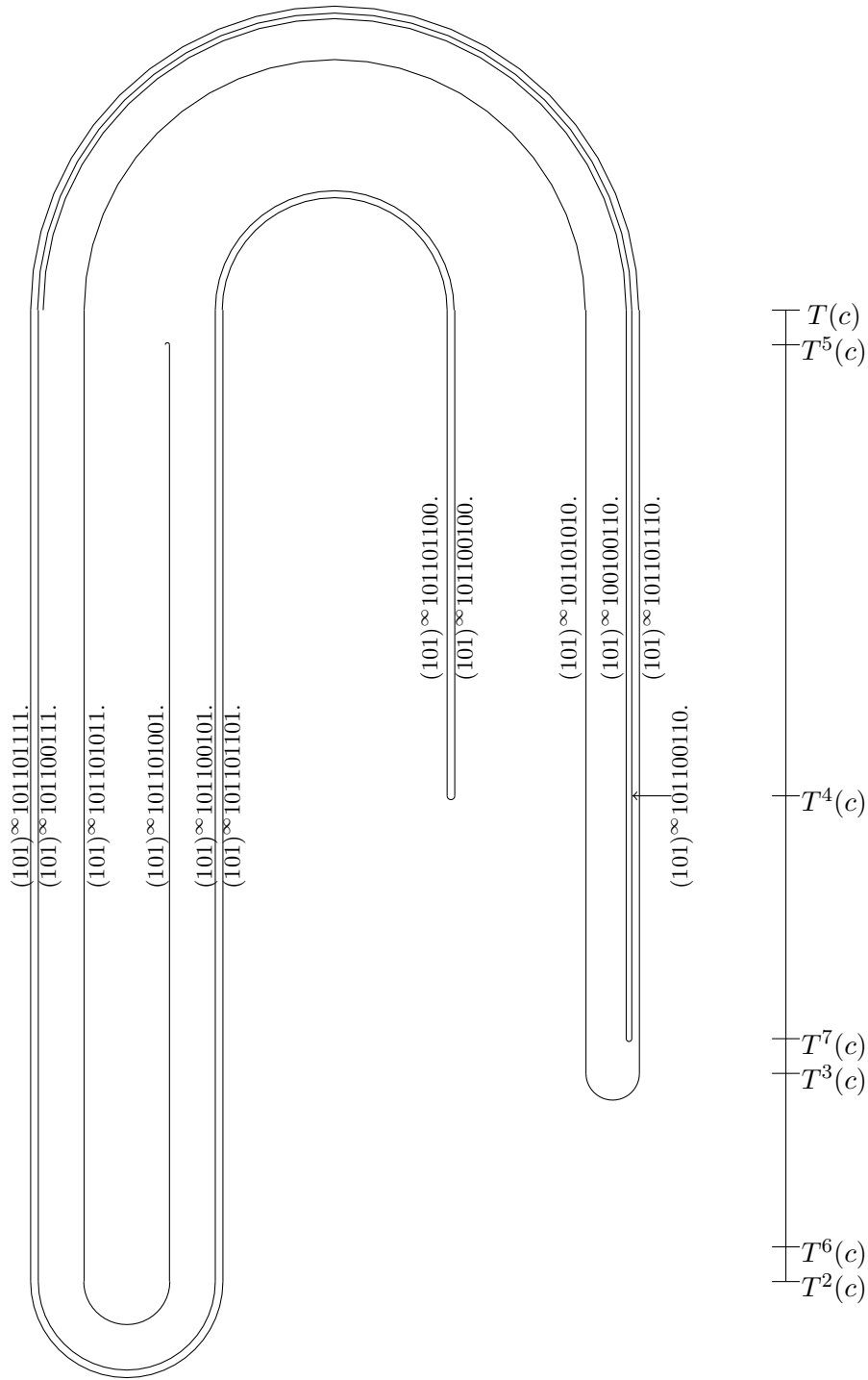
\begin{figure}[!ht]
\centering
\begin{tikzpicture}[scale=0.96,rotate=90]

\draw (0,-2)--(14,-2);
\draw (0,-2.2)--(0,-1.8); \node at (-0.1,-2.7){\small $T^2(c)$};
\draw (0.5,-2.2)--(0.5,-1.8); \node at (0.4,-2.7){\small $T^6(c)$};
\draw (3,-2.2)--(3,-1.8); \node at (2.9,-2.7){\small $T^3(c)$};
\draw (3.5,-2.2)--(3.5,-1.8); \node at (3.4,-2.7){\small $T^7(c)$};
\draw (7,-2.2)--(7,-1.8); \node at (6.9,-2.7){\small $T^4(c)$};
\draw (13.5,-2.2)--(13.5,-1.8); \node at (13.4,-2.7){\small $T^5(c)$};
\draw (14,-2.2)--(14,-1.8); \node at (13.9,-2.7){\small $T(c)$};

\draw (3,0.11)--(14,0.11);
\draw (3.5,0.22)--(14,0.22);
\draw (3.5,0.3)--(14,0.3);
\draw (3,0.88)--(14,0.88);
\draw (7,2.77)--(14,2.77);
\draw (7,2.88)--(14,2.88);
\draw (0,6.11)--(14,6.11);
\draw (0,6.22)--(14,6.22);
\draw (0,6.88)--(13.5,6.88);
\draw (0,8.11)--(14,8.11);
\draw (0,8.77)--(14,8.77);
\draw (0,8.88)--(14,8.88);

\draw[domain=270:450] plot ({14+4.38*cos(\x)}, {4.5+4.38*sin(\x)});
\draw[domain=270:450] plot ({14+4.28*cos(\x)}, {4.5+4.28*sin(\x)});
\draw[domain=270:450] plot ({14+4.2*cos(\x)}, {4.5+4.2*sin(\x)});
\draw[domain=270:450] plot ({14+3.61*cos(\x)}, {4.5+3.61*sin(\x)});
\draw[domain=270:450] plot ({14+1.72*cos(\x)}, {4.5+1.72*sin(\x)});
\draw[domain=270:450] plot ({14+1.62*cos(\x)}, {4.5+1.62*sin(\x)});
\draw[domain=270:450] plot ({13.5+0.03*cos(\x)}, {6.91+0.03*sin(\x)});

\draw[domain=90:270] plot ({0.62*cos(\x)}, {7.495+0.62*sin(\x)});
\draw[domain=90:270] plot ({1.275*cos(\x)}, {7.495+1.275*sin(\x)});
\draw[domain=90:270] plot ({1.385*cos(\x)}, {7.495+1.385*sin(\x)});
\draw[domain=90:270] plot ({7+0.055*cos(\x)}, {2.825+0.055*sin(\x)});
\draw[domain=90:270] plot ({3+0.385*cos(\x)}, {0.495+0.385*sin(\x)});
\draw[domain=90:270] plot ({3.5+0.04*cos(\x)}, {0.26+0.04*sin(\x)});

\node[rotate=90] at (7,9){\scriptsize $(101)^{\infty}101101111.$};
\node[rotate=90] at (7,8.6){\scriptsize $(101)^{\infty} 101100111.$};
\node[rotate=90] at (7,7.9){\scriptsize $(101)^{\infty} 101101011.$};
\node[rotate=90] at (7,7.05){\scriptsize $(101)^{\infty} 101101001.$};
\node[rotate=90] at (7,6.35){\scriptsize $(101)^{\infty} 101100101.$};
\node[rotate=90] at (7,5.95){\scriptsize $(101)^{\infty} 101101101.$};
\node[rotate=90] at (10,3.1){\scriptsize $(101)^{\infty} 101101100.$};
\node[rotate=90] at (10,2.6){\scriptsize $(101)^{\infty} 101100100.$};
\node[rotate=90] at (10,1.1){\scriptsize $(101)^{\infty} 101101010.$};
\node[rotate=90] at (10,-0.1){\scriptsize $(101)^{\infty} 101101110.$};
\node[rotate=90] at (10,0.5){\scriptsize $(101)^{\infty} 100100110.$};
\node[rotate=90] at (7,-0.5){\scriptsize $(101)^{\infty} 101100110.$};
\draw[->] (7,-0.35)--(7,0.21);

\end{tikzpicture}
\caption{The planar representation of an arc in $X$ with the corresponding kneading sequence $\protect\nu=100110010\ldots$. 
The ordering on basic arcs is such that the basic arc coded by $\protect L=1^{\infty}.$ is the largest. }
\label{representation1}
\end{figure}

\begin{figure}[!ht]
\centering
\begin{tikzpicture}[scale=0.96,rotate=90]

\draw (0,-2)--(14,-2);
\draw (0,-2.2)--(0,-1.8); \node at (-0.1,-2.7){\small $T^2(c)$};
\draw (0.5,-2.2)--(0.5,-1.8); \node at (0.4,-2.7){\small $T^6(c)$};
\draw (3,-2.2)--(3,-1.8); \node at (2.9,-2.7){\small $T^3(c)$};
\draw (3.5,-2.2)--(3.5,-1.8); \node at (3.4,-2.7){\small $T^7(c)$};
\draw (7,-2.2)--(7,-1.8); \node at (6.9,-2.7){\small $T^4(c)$};
\draw (13.5,-2.2)--(13.5,-1.8); \node at (13.4,-2.7){\small $T^5(c)$};
\draw (14,-2.2)--(14,-1.8); \node at (13.9,-2.7){\small $T(c)$};

\draw (7,0.11)--(14,0.11);
\draw (7,0.22)--(14,0.22);
\draw (3,2.11)--(14,2.11);
\draw (3.5,2.77)--(14,2.77);
\draw (3.5,2.83)--(14,2.83);
\draw (3,2.88)--(14,2.88);
\draw (0,6.11)--(14,6.11);
\draw (0,6.22)--(14,6.22);
\draw (0,6.88)--(14,6.88);
\draw (0,8.11)--(13.5,8.11);
\draw (0,8.77)--(14,8.77);
\draw (0,8.88)--(14,8.88);

\draw[domain=270:450] plot ({14+1.725*cos(\x)}, {4.5+1.725*sin(\x)});
\draw[domain=270:450] plot ({14+2.39*cos(\x)}, {4.5+2.39*sin(\x)});
\draw[domain=270:450] plot ({14+4.28*cos(\x)}, {4.5+4.28*sin(\x)});
\draw[domain=270:450] plot ({14+4.39*cos(\x)}, {4.5+4.39*sin(\x)});
\draw[domain=270:450] plot ({14+1.62*cos(\x)}, {4.5+1.62*sin(\x)});
\draw[domain=270:450] plot ({14+1.67*cos(\x)}, {4.5+1.67*sin(\x)});
\draw[domain=270:450] plot ({13.5+0.04*cos(\x)}, {8.15+0.04*sin(\x)});

\draw[domain=90:270] plot ({3.5+0.03*cos(\x)}, {2.8+0.03*sin(\x)});
\draw[domain=90:270] plot ({0.62*cos(\x)}, {7.495+0.62*sin(\x)});
\draw[domain=90:270] plot ({1.275*cos(\x)}, {7.495+1.275*sin(\x)});
\draw[domain=90:270] plot ({1.385*cos(\x)}, {7.495+1.385*sin(\x)});
\draw[domain=90:270] plot ({7+0.055*cos(\x)}, {0.165+0.055*sin(\x)});
\draw[domain=90:270] plot ({3+0.385*cos(\x)}, {2.495+0.385*sin(\x)});

\node[rotate=90] at (7,9){\scriptsize $(101)^{\infty}101101101.$};
\node[rotate=90] at (7,8.6){\scriptsize $(101)^{\infty} 101100101.$};
\node[rotate=90] at (7,7.9){\scriptsize $(101)^{\infty} 101101001.$};
\node[rotate=90] at (7,7.05){\scriptsize $(101)^{\infty} 101101011.$};
\node[rotate=90] at (7,6.35){\scriptsize $(101)^{\infty} 101100111.$};
\node[rotate=90] at (7,5.95){\scriptsize $(101)^{\infty} 101101111.$};

\node[rotate=90] at (9,3.1){\scriptsize $(101)^{\infty} 101101110.$};
\draw[->] (7,3.8)--(7,2.83);
\node[rotate=90] at (7,4){\scriptsize $(101)^{\infty} 100100110.$};
\node[rotate=90] at (9,2.6){\scriptsize $(101)^{\infty} 101100110.$};
\node[rotate=90] at (9,1.9){\scriptsize $(101)^{\infty} 101101010.$};
\node[rotate=90] at (10.5,0.4){\scriptsize $(101)^{\infty} 101100100.$};
\node[rotate=90] at (10.5,-0.1){\scriptsize $(101)^{\infty} 101101100.$};

\end{tikzpicture}
\caption{The planar representation of the same arc as in Figure~\ref{representation1} in $X$ with the corresponding kneading sequence $\protect\nu=100110010\ldots$. 
The ordering on basic arcs is such that the basic arc coded by $\protect L=(101)^{\infty}.$ is the largest.}
\label{representation2}
\end{figure}
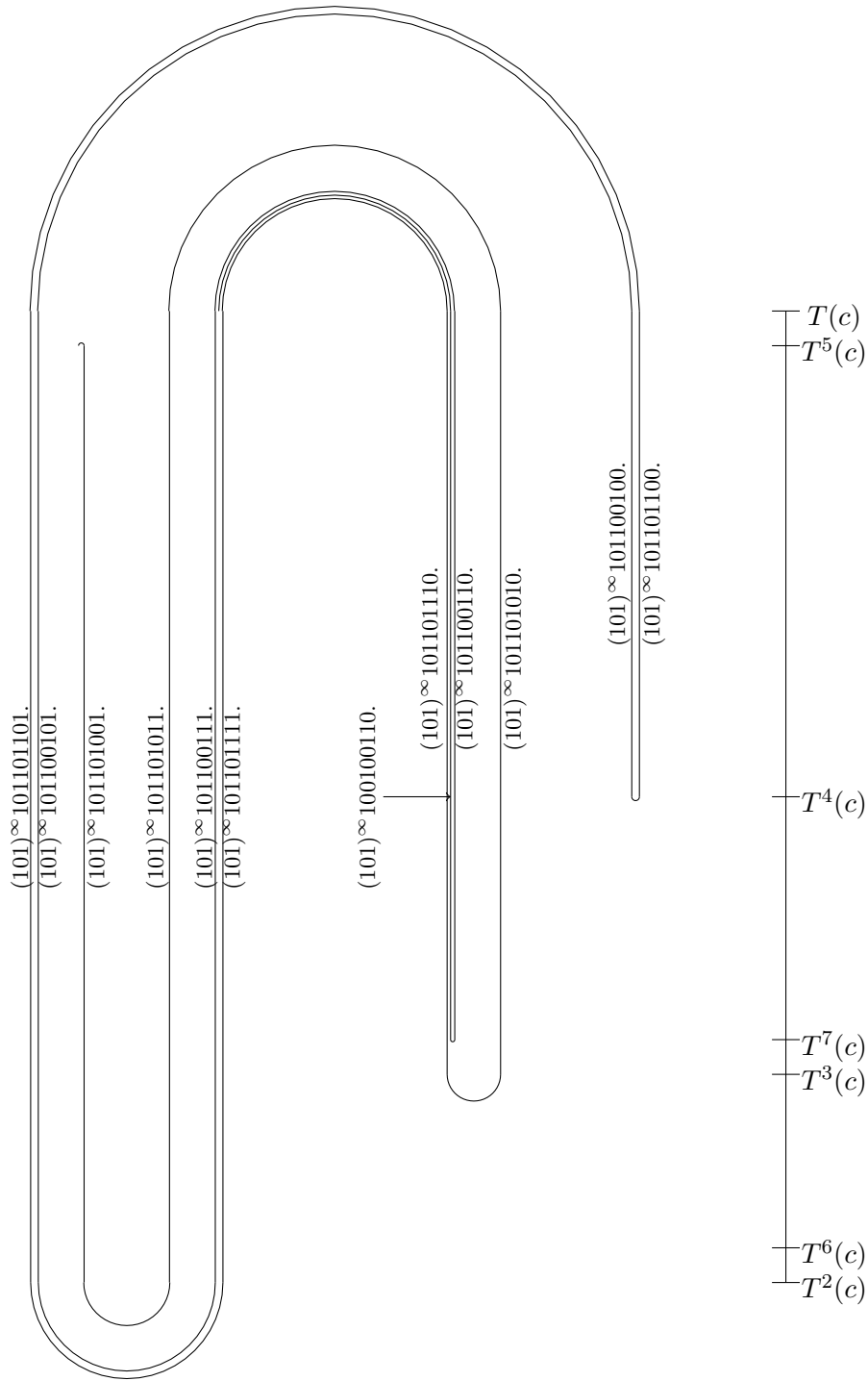

\section{Embeddings}\label{sec:embeddings}
In this section we show that representations of $X$ constructed in the previous section are indeed embeddings.

\begin{lemma}\label{lem:fourth}
Let $U\subset\R^2$ be homeomorphic to the open unit disk, and let $W\subset\R$ be a closed set such that $W \times J \subset U$
for some closed interval $J$.
There exists a continuous function $f:\R^2\to\R^2$ such that
$f(\{w\}\times J)$ is a point for every $w\in W$, 
$f(\{w\}\times J) \neq f(\{w'\}\times J)$ for every $w\neq w'\in W$, 
$f|_{U\setminus W\times J}$ is injective and $f|_{\R^2\setminus U}$ is the identity.
\end{lemma}

\begin{proof}
Without loss of generality we can take $U:=(-1,2)\times (-1,1)$, $J:=[-1/2, 1/2]$ and $\min(W)=0$, $\max(W)=1$, see Figure~\ref{fig:glue}.

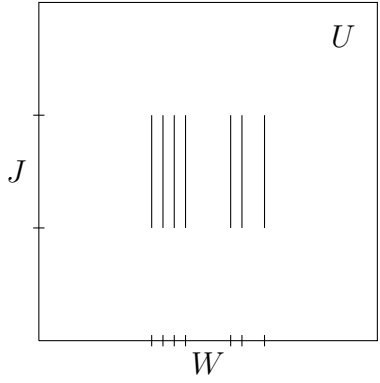
\begin{figure}[!ht]
	\centering
	\begin{tikzpicture}[scale=1.5]
	\draw (0,0)--(0,3)--(3,3)--(3,0)--(0,0);
	\draw (-0.05,1)--(0.05,1);
	\draw (-0.05,2)--(0.05,2);
	\node at (-0.2,1.5) {$J$};
	\node at (1.5,-0.2) {$W$};
	\draw (1,-0.05)--(1,0.05);
	\draw (2,-0.05)--(2,0.05);
	\draw (1.1,-0.05)--(1.1,0.05);
	\draw (1.2,-0.05)--(1.2,0.05);
	\draw (1.3,-0.05)--(1.3,0.05);
	\draw (1.7,-0.05)--(1.7,0.05);
	\draw (1.8,-0.05)--(1.8,0.05);
	\node at (2.7,2.7) {$U$};
	\draw (1,1)--(1,2);
	\draw (1.1,1)--(1.1,2);
	\draw (1.2,1)--(1.2,2);
	\draw (1.3,1)--(1.3,2);
	\draw (1.7,1)--(1.7,2);
	\draw (1.8,1)--(1.8,2);
	\draw (2,1)--(2,2);
	\end{tikzpicture}
	\caption{Set-up in Lemma~\ref{lem:fourth}.}
	\label{fig:glue}
\end{figure}

For every $a\in[0,2]$ we define a continuous function $g(a,\cdot) : [-1,1]\to [-1,1]$ as
$$
g(a,x):=\left\{
\begin{array}{rl}
(2-a)x+1-a, & x\in[-1, -1/2]\\
ax, & x\in[-1/2, 1/2]\\
(2-a)x+a-1, & x\in[1/2, 1].
\end{array}
\right.
$$
Note that $g(a, \cdot )$ is injective for every $a\in[0, 2]$, $g(0, x)=0$ for all $x\in [-1/2, 1/2]$, and $g(1, x)=x$ for all $x\in[-1,1]$.

Define $\hat f:[-1, 2]\times[-1,1]\to[-1, 2]\times[-1,1]$ as
$$\hat f(x,y):=(x, g (d_e(x, W), y) ),$$
where $d_e(x, W)=\inf_{w \in W}\{d_e(x, w)\}$. Note that $x\mapsto d_e(x, W)$ is continuous, so $\hat f$ is continuous. 
Also, $\hat f(w,y)=(w, g(0,y))=(w, 0)$ for $(w,y)\in W\times J$ 
and $\hat f$ is injective otherwise. Also note that $\hat f$ is the identity on the boundary of $[-1, 2]\times [-1, 1]$, so $\hat f$ can be extended continuously to the map $f:\R^2\to\R^2$ such that
$f|_{\R^2\setminus U}$ is the identity.
\end{proof}

Define $W_n\subset\R^2$ to be the set consisting of all semi-circles that join pairs of points at level $n$.
Note that there exists a set $W\subset\R$ such that $W_n$ is homeomorphic to $W\times J$. 
Observe that $W$ is closed. Indeed, if for a sequence $(\overleftarrow{s}^{k})_{k\in\N}\subset \{ 0,1\}^{\N}$
there exists $m\in\N$ such that $\tau_{R}(\overleftarrow{s}^{k})
=m$ for every $k\in\N$ and $\lim_{k\rightarrow \infty}\overleftarrow{s}^{k}=\overleftarrow{s}$, then 
$\tau_{R}(\overleftarrow{s})=m$. The analogous argument holds 
for $\tau_L$.

\begin{lemma}
There exist open sets $U_n\subset\R^2$ such that $W_n\subset U_n$ and for every $n\neq m\in \N$, $U_n\cap U_m=\emptyset$ and $\diam(U_n)\to 0$ as $n\to \infty$.
\end{lemma}

\begin{proof}
We define the set $G_n:=\{\psi_L(\overleftarrow{s}) : \overleftarrow{s}\in\{0, 1\}^{-\N} \text{ admissible},\  \tau_R(\overleftarrow{s})=n \text{ or }\\ \tau_L(\overleftarrow{s})=n\}$ for
every $n\in \N$ and let $A_n$ be the smallest interval in $[0,1]$ 
containing $G_n$. 
Note that $A_n$ is closed and $\diam(A_n)\leq 3^{-n}$.

Let $M_n$ denote the midpoint of $A_n$.
If $\#_1(c_1\ldots c_n)$ is odd, let
$$
V'_n=\left\{(x,y)\in\R^2 : (x-T^n(c))^2+(y-M_n)^2\leq \left(\frac{\diam(A_n)}{2}\right)^2,\ x\leq T^n(c)\right\}
$$
be the closed left semi-disc centered around $(T^n(c), M_n)$. 
Similarly, if $\#_1(c_1\ldots c_n)$ is even, let
$$
V'_n=\left\{(x,y)\in\R^2 : (x-T^n(c))^2+(y-M_n)^2\leq \left(\frac{\diam(A_n)}{2}\right)^2,\ x\geq T^n(c)\right\}.
$$
be the closed right semi-disc centered around $(T^n(c), M_n)$. 
Note that $W_n\subset V'_n$, $\diam(V'_n)\leq 3^{-n}$
and that $d_e(A_n, \psi_L(\overleftarrow{t}))>3^{-n}$ for all $\psi_L(\overleftarrow{t})\notin A_n$. Let $V_n$ be the
$\frac{\diam(A_n)}{2\cdot 3}$-neighbourhood
of $V'_n$, that is, 
$$
V_n=\left\{x\in\R^2 : \textrm{ there exists } y \in V'_n\textrm{ such that }d_e(x, y)<\frac{\diam(A_n)}{2\cdot 3}\right\},
$$
see Figure~\ref{semi-circ}.
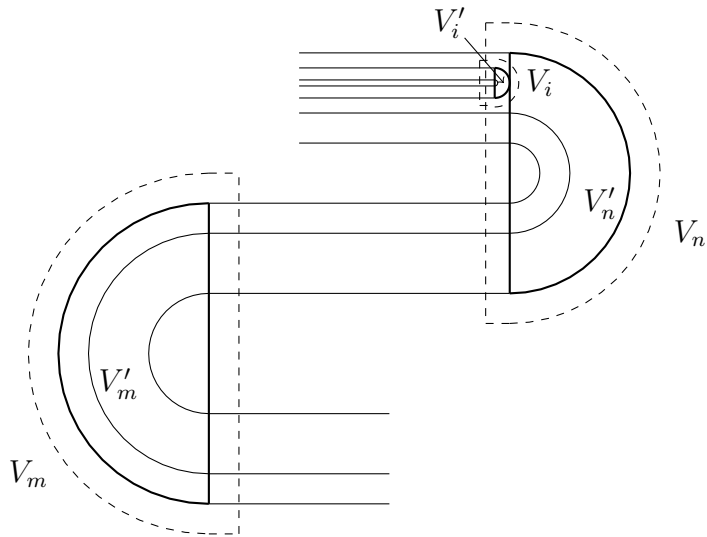
\begin{figure}[!ht]
\centering
\begin{tikzpicture}[scale=0.4]

\draw (10,1)--(16,1);
\draw (10,2)--(16,2);
\draw (10,4)--(16,4);
\draw (10,8)--(20,8);
\draw (10,10)--(20,10);
\draw (10,11)--(20,11);
\draw (13,13)--(20,13);
\draw (13,14)--(20,14);
\draw (13,16)--(20,16);

\draw[domain=270:450] plot ({20+1*cos(\x)}, {12+1*sin(\x)});
\draw[domain=270:450] plot ({20+2*cos(\x)}, {12+2*sin(\x)});
\draw[thick,domain=270:450] plot ({20+4*cos(\x)}, {12+4*sin(\x)});
\draw[thick](20,16)--(20,8);
\draw[dashed,domain=270:450] plot({20+5*cos(\x)}, {12+5*sin(\x)});
\draw[dashed] (19.2,17)--(20,17);
\draw[dashed] (19.2,7)--(20,7);
\draw[dashed](19.2,17)--(19.2,7);
\draw[domain=90:270] plot ({10+2*cos(\x)}, {6+2*sin(\x)});
\draw[domain=90:270] plot ({10+4*cos(\x)}, {6+4*sin(\x)});
\draw[thick,domain=90:270] plot ({10+5*cos(\x)}, {6+5*sin(\x)});
\draw[thick](10,11)--(10,1);

\draw(13,14.5)--(19.5,14.5);
\draw (13,14.9)--(19.5,14.9);
\draw(13,15.1)--(19.5,15.1);
\draw(13,15.5)--(19.5,15.5);

\draw[thick,domain=270:450] plot ({19.5+0.5*cos(\x)}, {15+0.5*sin(\x)});
\draw[thick](19.5,15.5)--(19.5,14.5);
\draw[domain=270:450] plot ({19.5+0.1*cos(\x)}, {15+0.1*sin(\x)});
\draw[dashed,domain=270:450] plot({19.5+0.8*cos(\x)},{15+0.8*sin(\x)});
\draw[dashed] (19,15.75)--(19.5,15.75);
\draw[dashed](19,14.25)--(19.5,14.25);
\draw[dashed](19,15.75)--(19,14.25);

\draw[dashed,domain=90:270] plot({10+6*cos(\x)}, {6+6*sin(\x)});
\draw[dashed] (11,0)--(11,12);
\draw[dashed](10,12)--(11,12);
\draw[dashed](10,0)--(11,0);

\node at (4,2) {$V_m$};
\node at (7,5) {$V'_m$};
\node at (26,10) {$V_n$};
\node at (23,11){$V'_n$};
\node at (21,15){$V_i$};
\node at (18,17){$V'_i$};
\draw[->] (18.5,16.5)--(19.8,15);
\end{tikzpicture}
\caption{Sets constructed in the proof of Lemma~\ref{lem:fourth}.}
\label{semi-circ}
\end{figure}
For every $n\in \N$, the open set
$$U_n:=V_n\setminus\overline{\cup_{i>n}V_i}$$
contains $W_n$,
because otherwise there exists an increasing sequence 
$(i_k)_{k\in\N}\subset \N$ so that points $x^{i_k}\in \{T^{i_k}(c)\}\times G_{i_k}$ 
and $x:=\lim_k x^{i_k}\in W_n$. Since $x^{i_k}\in\{T^{i_k}(c)\}\times G_{i_k}$, 
the corresponding itinerary
satisfies $\tau_R(\overleftarrow{s}^{i_k})=i_k$, but because $i_k\rightarrow \infty$ as $k\rightarrow \infty$ this implies that the corresponding itinerary $\overleftarrow{s}$ of $x$ satisfies 
$\tau_R(\overleftarrow{s})=\infty$, a contradiction.

Note that $\diam(U_n)\leq\diam(V_n)\to 0$ as $n\to\infty$.
\end{proof}

Now define a continuous function $f_n\colon\R^2\to\R^2$ as 
in Lemma~\ref{lem:fourth} replacing $U$ with $U_n$ and $W$ with $W_n$.
Let $F_n:\R^2\rightarrow \R^2$ be defined as $F_n:=f_n\circ\ldots\circ f_1$ for every $n\in \N$. We need to show that 
$F := \lim_{n\rightarrow \infty} F_n$ exists and is continuous. It is enough to show the following:

\begin{lemma}
Sequence $(F_n)_{n\in \N}$ is uniformly Cauchy.
\end{lemma}

\begin{proof}
Take $n<m\in\N$ and note that $\sup_{x\in\R^2}d_e(F_m(x), F_n(x))=\sup_{x\in\R^2}d_e(f_m\circ\ldots\circ f_{n+1}\circ F_n(x), F_n(x))<\max\{\diam(U_{n+1}), \ldots , \diam(U_{m})\} \to 0$ as $n, m\to\infty$.
\end{proof}

Denote by $Z:=Y\cup \{\textrm{semi-circles}\}\subset\R^2$. We want to argue that $F(Z)\subset\R^2$ is 
homeomorphic to $Y/\!\!\sim$. 
Since $F\colon Z\to F(Z)$ is continuous, it follows from \cite[Theorem 3.21]{Na}, that $\{F^{-1}(y) : y\in F(Z)\}$ 
is a decomposition of $Z$ homeomorphic to $F(Z)$. Note that this decomposition is exactly $Y/\!\!\sim$.

\begin{figure}[!ht]
	\centering
	\begin{tikzpicture}[scale=1.5]
	\draw[solid] (1,1) circle (1);
	\draw (-0.5,1)--(2.5,1);
	\draw (-0.5,0.5)--(2.5,0.5);
	\draw (-0.5,0.75)--(2.5,0.75);
	\draw (-0.5,0.88)--(2.5,0.88);
	\draw (-0.5,0.94)--(2.5,0.94);
	\draw (-0.5,0.96)--(2.5,0.96);
	\draw[thick] (1, 1)--(1,1.5);
	\node at (1.1,1.35){\small $A$};
	\node[circle,fill, inner sep=1] at (1,1){};
	\node at (0.9,1.1){\small $a$};
	
	\end{tikzpicture}
	\caption{Point $a = (a_0, \psi_L(L))$  is accessible.}
	\label{fig:top}
\end{figure}
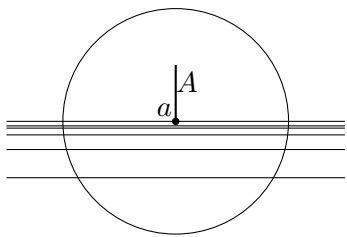

\begin{proof}[Proof of Theorem~\ref{thm:main}]
	Assume that the symbolic representation of $a=(\ldots, a_{-2}, a_{-1}, a_0)\\ \in X$ is given by $\bar{I}(a)=\ldots l_{-2}l_{-1}.l_0 l_1\ldots$. 
	Consider the planar representation $Z$ of $X$ obtained by the ordering on $C$ making $L=\ldots l_{-2}l_{-1}$ the largest. The point $a$ is represented as $(a_0, 1)$. 
	Take the arc $A=\{(a_0, t+1), t\in[0,1]\}$ which is a vertical 
	interval in the plane (see Figure~\ref{fig:top}). Note that $A\cap Z=\{a\}$. Then $F(A)$ is an arc such that $F(A)\cap F(Z)=\{F(a)\}$
	which concludes the proof. 
\end{proof}

From now onwards we denote the \emph{core inverse limit space} $\underleftarrow\lim([T^{2}(c), T(c)], T)$ by $X'$.
\begin{lemma}\label{lem:indec}
	Let $T$ be a unimodal map such that $s := \exp(h_{top}(T))> \sqrt{2}$.
	Then the core inverse limit space $X'$ is indecomposable.
\end{lemma}

\begin{proof}
	The map $T$ is semiconjugate to the tent map $T_s$, and if $T$ is locally eventually onto (leo),
	then the semiconjugacy $h$ is in fact a conjugacy. In this case, $X'$ is indecomposable, see \cite{InMa}.
	We give the argument if $T$ is not leo (which also works in the general case).
	Let $p$ be the orientation reversing fixed point of $T$, so $h(p) = \frac{s}{s+1}$ is the fixed point of $T_s$.
	Let $J \ni p$ be a neighbourhood
	such that $h(J)$ is a non-degenerate neighbourhood of $\frac{s}{s+1}$.
	Since $s > \sqrt{2}$, $T_s$ is leo, so there is $N\in\N$ such that
	$T^N(\overline{J_-}) = T^N(\overline{J_+}) = [T^{2}(c), T(c)]$ for both components $J_\pm$ of $J \setminus \{ p \}$.
	Suppose now by contradiction that $X'= A \cup B$ for some proper subcontinua $A$ and $B$ of $X'$.
	Hence there exists $n_0\in \N$ such that the projections $\pi_{n_0}(A) \neq [T^{2}(c), T(c)] \neq \pi_{n_0}(B)$.
	Take $n_1 = n_0 + N$. Since $\pi_{n_1}(A)$ and $\pi_{n_1}(B)$ are intervals
	and $\pi_{n_1}(A) \cup \pi_{n_1}(B) = [T^{2}(c), T(c)]$, at least one of them, say 
	$\pi_{n_1}(A)$, contains at least one of $\overline{J_-}$ or $\overline{J_+}$.
	But then $\pi_{n_0}(A) \supset T^N(\overline{J_-}) \cap T^N(\overline{J_+}) = [T^{2}(c), T(c)]$, contradicting
	the definition of $n_0$. This completes the proof.
\end{proof}

\begin{proof}[Proof of Corollary~\ref{cor:equiv}]
	First assume that $\exp(h_{top}(T))> \sqrt{2}$.
	By Lemma~\ref{lem:indec}, $X'$ is indecomposable and thus it has uncountably many pairwise disjoint composants which are dense in $X'$. By Proposition 2 from \cite{BB} every subcontinuum $H\subset X'$ contains a ray which is dense in $H$. Therefore, every composant $\mathcal{U}$ of $X'$ contains a non-degenerate basic arc in $X'$. We embed $X$ so that this non-degenerate basic arc is the largest. Such embedding of $X$ makes a non-degenerate arc of $\mathcal{U}$ accessible.\\	
	Assume that $g_1\colon X\to E_1$ and $g_2\colon X\to E_2$ are equivalent embeddings, so there exists a homeomorphism $\tilde{h}:\mathbb{R}^2\rightarrow \mathbb{R}^2$ such that $\tilde{h}(E_1)=E_2$. 
	It was proven in \cite{BBS} that every homeomorphism $h\colon X\to X$ is pseudo-isotopic to $\sigma^R$ for some $R\in\Z$. This means that $h$ permutes the composants of $X'$ in the same way as $\sigma^R$ does. We apply this to $\tilde{h}=g_2\circ h\circ g_1^{-1}$.
	Thus if $p\in E_1$ is an accessible point, then $q:=g_2\circ h\circ g_1^{-1}(p)$ is also accessible, and $g_2\circ\sigma^{R}\circ g_1^{-1}(p)$ belongs to the same composant of $X'$ as $q$. 
	Hence $h$ maps a composant of $X'$ with an accessible non-degenerate arc to a composant of $X'$ with an accessible non-degenerate arc. 
	 Mazurkiewicz \cite{Maz} proved that every indecomposable planar continuum contains at most countably many composants with an accessible non-degenerate arc. Combining this with the fact that there are uncountably many composants in $X'$ that are not shifts of one another, we finish the proof in case when $\exp(h_{top}(T))> \sqrt{2}$. \\
	 Now assume that $\sqrt{2}\geq \exp(h_{top}(T))>1$. The core $X'$ is decomposable and there exists an indecomposable subcontinuum of $X'$ which is homeomorphic to the inverse limit space of a unimodal map with entropy greater 
	 than $\log \sqrt 2$. It follows from the arguments above that we can embed this indecomposable subcontinuum in uncountably many non-equivalent ways; therefore we obtain uncountably many non-equivalent embeddings of $X$.
\end{proof}
\begin{remark}
 Mazurkiewicz' result is perhaps too strong a tool to apply, but to complete the argument
 we need to know which points are accessible besides those in (the composant of) the
 arc which is made the largest. Frequently, this composant is indeed the only accessible
 composant, but there are many exceptions. Determining which points are accessible is a nontrivial
 matter, to which we plan to come back in a forthcoming paper. 
\end{remark}

\end{document}